\documentclass[a4paper,12pt]{amsart}

\usepackage{amsmath}
\usepackage{amsfonts}
\usepackage{graphicx}
\usepackage[abbrev]{amsrefs}
\usepackage{amsthm}
\usepackage{color}
\usepackage[all,cmtip]{xy}

\newtheorem{theorem}{Theorem}[section]
\newtheorem{definition}[theorem]{Definition}
\newtheorem{lemma}[theorem]{Lemma}
\newtheorem{proposition}[theorem]{Proposition}

\newtheorem{remark}[theorem]{Remark}

\newcommand{\bP}{\mathbb{P}}

\title[Connectedness Bertini Theorem]
{Connectedness Bertini Theorem via numerical equivalence}

\author[Martinelli]{Diletta Martinelli}

\address{Department of Mathematics, Imperial College London, 180 Queen's
Gate, London SW7 2AZ, UK.}
\email{d.martinelli12@imperial.ac.uk}

\author[Naranjo]{Juan Carlos Naranjo}
\address{Departament d'\`Algebra i Geometria,
Facultat de Matemàtiques,
Universitat de Barcelona,
Gran Via, 585
08007 Barcelona, Spain.}
\email{ jcnaranjo@ub.edu}

\author[Pirola]{Gian Pietro Pirola}
\address{Dipartimento di Matematica,
Univerista' di Pavia,
Via Ferrata 1,
27100 Pavia, Italy}
\email{gianpietro.pirola@unipv.it}

\begin{document}

\begin{abstract} Let $X$ be an irreducible projective variety and $f$ a morphism $X \rightarrow \mathbb{P}^n$. We give a new proof of the fact that the preimage of any linear variety of dimension $k\ge n+1-\dim f(X)$ is connected. We show that the statement is a consequence of the Generalized Hodge Index Theorem using easy numerical arguments that hold in any characteristic. We also prove the connectedness Theorem of Fulton and Hansen as application of our main theorem.
\end{abstract}


\maketitle

\section{Introduction}

In this note we focus our attention on Bertini Theorem. Eugenio Bertini (1846-1933) was a student of Luigi Cremona, one of the founders of the Italian school of Algebraic Geometry, who lived and worked in Pavia from 1880 to 1892. In his paper \cite{bertini} dated 1880 and published in 1882, he proved that for a nonsingular projective variety $X \subseteq \mathbb{P}^n$ the general hyperplane section is nonsingular, that is, there is a non empty Zariski open set $U \subseteq (\mathbb{P}^{n}) ^{\vee}$ such that for every $H \in U$,  the subvariety $H \cap X$ is regular at every point. Moreover, if the dimension of the variety is at least two, the general hyperplane section is connected.  For more details on the life and work of Bertini and on the history of this theorem we refer the reader to the paper of Kleiman \cite{Kleiman}. This famous theorem has been generalized in many directions, especially in the context of the theory of linear series. For further details, see the classic book of Jouanolou \cite{jouanolou}.

We consider here a slightly more general statement: let $X$ be an irreducible
projective variety defined over an algebraically closed field of any characteristic and let
\begin{center}
$f: X \longrightarrow \mathbb{P}^n$
\end{center}
be a morphism. Then the preimage of a linear variety $L \subseteq \mathbb{P}^n$ is connected, if $\dim L + \dim f(X) > n$.

Over the complex numbers, the statement is usually proved using the Generic Smoothness Theorem, that fails in positive characteristic (see for instance \cite[ Corollary 10.7]{hartshorne} for the statement of the Generic Smoothness and \cite[Theorem 3.3.1, Theorem 3.3.3]{lazarsfeld} for a proof of Bertini Theorem in this case).
A characteristic free proof of the irreducibility is \cite[Theorem 17, IX.6]{Weil} and \cite[Theorem 9]{seidenberg}.  Our aim is to give a new direct proof of the connectedness statement without using the Generic Smoothness Theorem. The main interest of our approach is the use of numerical connectedness to deduce topological connectedness on the preimages of linear varieties. In particular, we show that the connectedness result is a consequence of the Generalized Hodge Index Theorem. Hence, our proof works in any characteristic. 
 
More precisely, we prove the following.

\begin{theorem}\label{main}
Let $X$ be an irreducible projective variety and let 
\begin{center}
$f: X \longrightarrow \mathbb{P}^n$
\end{center}
be a morphism. Then for any linear subvariety $L \in \mathbb{G}(k,n)$ of dimension $k\ge n+1-\dim f(X)$, the preimage $f^{-1}(L)$ is connected.
\end{theorem}

The steps of our proof are the following: first, in  Section \ref{dominantgen}, we prove that we can assume that $f$ is dominant and the linear variety $L$ is general in the corresponding Grassmannian. Then in Section \ref{line} we prove the statement when $k=1$: first, we assume that $n=2$ and we use the fact that big and nef divisors are $1$-connected, see Proposition \ref{1conn}, and then we generalize to any $n$. Finally, we conclude with a proof by induction on the dimension of $L$ (Section \ref{higherdim}). 

In Section \ref{fultonhansen}, following an idea of Deligne, we deduce the connectedness Theorem of Fulton and Hansen \cite{fulthansen}. This theorem is a striking generalization of Bertini Theorem with many interesting geometric applications (see \cite{fultlaz}). For instance, F.L. Zak, in \cite{zak}, has used the connectedness theorem to establish a famous result on tangencies to a smooth subvariety $X \subseteq \mathbb{P}^m$ of dimension $n$, from which he was able to deduce that if $3n > 2(m - 1)$ then $X$ is linearly normal, as predicted by the Hartshorne's conjecture. Morover, over the complex number, it is possible to obtain information about the fundamental groups by applying these connectivity results to a covering space of the variety. In this context, Deligne, see \cite{Deligne}, generalized the connectedness theorem to a statement about the fundamental group $\pi_1$ and $\pi_0$, and then his work has been extended to higher homotopy groups, see \cite[Section 9]{fultlaz} and \cite{sommese}. It would be interesting to see whether using our methods some of these problems might be generalized to the algebraic setting. Many nice questions remain open. 

\vspace{5mm}

{\bf Acknowledgments:} This article has been originally developed during the 
undergraduate thesis of the first author at the University of Barcelona and at the University of Pavia in a joint work with the second and the third author. We would like to thank Paolo Cascini, Andrew Sommese and Roland Abuaf for their useful comments and suggestions. The first author is supported by a Roth scholarship. The second author has been partially supported by the Proyecto de Investigaci\'on MTM2012-38122-C03-02; The third author has been partially supported by Gnsaga and by MIUR PRIN 2012: \textup{Moduli, strutture geometriche e loro applicazioni}.

\section{Preliminaries}
In this paper all the varieties are considered with the Zariski topology. We recall some definitions.

\begin{definition}
Let $X$ be a complete variety and $D$ a divisor on $X$, we say that $D$ is nef (numerically effective) if $D \cdot C \geq 0$, for all irreducible curves $C \subseteq X$. 
\end{definition} 

\begin{definition}
A line bundle $L$ on an irreducible projective variety $X$ is big if it has maximal Kodaira dimension: $k(X, L) = \dim X$.
\end{definition}

\begin{definition}
A line bundle $L$ on an irreducible projective variety $X$ is semiample if there exists an integer $r$ such that $L^{\otimes r}$ is globally generated.
\end{definition}

We will use the following characterization of bigness for nef divisors, see \cite[I, Theorem 2.2.16]{lazarsfeld}.
\begin{theorem} \label{bignessfornef}
Let $D$ be a nef divisor on an irreducible projective variety $X$ of dimension $n$. Then $D$ is big if and only if its top self-intersection is stricly positive: $D^n > 0$.
\end{theorem}

A central tool in the proof of numerical connectedness on surfaces is the following version of the Hodge Index Theorem, which is more general than the usual one for surfaces and can be easily deduced from the standard version (see \cite[V, Theorem 1.9] {hartshorne}):

\begin{theorem}\label{HIT}(Hodge Index Theorem) Let $S$ be a smooth projective surface and 
let $H$ be a divisor with $H^2 > 0$. Let $D$ be a divisor such that $D\cdot H = 0$. Then either $D^2 < 0$ or $D$ is numerically trivial.
\end{theorem}

Generalizations of this classical result have arisen in many directions. We will use the following inequality, see \cite[Theorem 1.6.1, Formula (1.24)]{lazarsfeld}:
\begin{theorem}[Generalized inequality of Hodge type] \label{genhodge}
 Let $X$ be an irreducible complete variety of dimension $n$, and let 
\begin{equation*}
\beta_1, \dots, \beta_{n-1}, h 
\end{equation*}
be numerical classes of nef divisors. Then 
\begin{equation*}
\big(\beta_1 \cdot \dots \cdot \beta_{n-1}\cdot h \big)^{n-1} \geq \big( (\beta_1)^{n-1}\cdot h \big) \cdots  \big( (\beta_{n-1})^{n-1}\cdot h \big).
\end{equation*}
\end{theorem}

One of the main problems in positive characteristic is the lack of resolution of singularities for varieties of dimension greater than three. Instead of the resolution of singularities, we will use several times the existence of alterations.
\begin{definition} \label{alteration} (See \cite{DeJong})
Let $X$ be a variety over a field $k$ algebraically closed. An alteration of X is a proper dominant morphism $X' \longrightarrow X$ of varieties over $k$, with $\dim X' = \dim X$. An alteration is regular if $X'$ is smooth.
\end{definition}
Thanks to \cite[Theorem 3.1]{DeJong}, there always exists a regular alteration.

\begin{remark}\label{irr-smooth}
In Theorem \ref{main} we can assume $X$ to be nonsingular. Indeed, assume by contradiction that there exists a variety $X$, a morphism $f: X \longrightarrow \mathbb {P}^n$ and a linear subvariety $L$ such that $f^{-1}(L)$ is disconnected. Consider a regular alteration $a: \tilde X \longrightarrow X$. Then the components of $a^{-1}f^{-1}(L)$ are disconnected. 
\end{remark}

\section{The main Theorem} \label{mainsection}
This section is devoted to the proof of Theorem \ref{main}. Observe that for $n=1$ there is nothing to prove because the only possible choice is $L = \mathbb{P}^1$, so we can assume $n\ge 2$. Thanks to Remark \ref{irr-smooth}, we can assume $X$ to be nonsingular.

\subsection{Reduction to $f$ dominant and $L$ general} \label{dominantgen}
It is enough to prove the theorem for a general linear variety. This is a consequence of an argument of Jouanolou (see the proof of \cite[Theorem 3.3.3]{lazarsfeld}). So from now on we will assume the generality of $L$.

Let $f: X \longrightarrow \mathbb{P}^n$ be a morphism and  $L \subseteq \mathbb {P}^{n}$ a linear variety, we assume that $k=\dim L\ge n+1-\dim f(X)$.
If $f$ is not dominant and $L$ is general, $L$ is not contained in $f(X)$. Then there exist a point $p \in L\setminus f(X)$. We project from the point $p$ to an hyperplane $H=\mathbb {P}^{n-1}$ and let $f':X\longrightarrow \mathbb{P}^{n-1}$ be the composition of $f$ with this projection. The intersection  $L\cap H$ is a linear subvariety  $L'\subset H$ such that
\begin{equation*}
\dim L' = k-1 \ge \dim H +1 - \dim f'(X)
\end{equation*} 
(observe that $\dim f'(X)=\dim f(X)$). If the theorem is true for the map $f'$, we have that $f'^{-1}(L')=f^{-1}(L)$ is connected. Therefore, it is enough to prove the theorem for $f'$. If $f'$ is not dominant, we perform successive projections until we reach a dominant map. So we can assume for the rest of the proof that $f$ is dominant.

\subsection{Connectedness of the preimage of a line by a generically finite map} \label{line}

The aim of this subsection is to prove the following particular case of Theorem \ref{main}.
\begin{proposition} \label{line:conn}
Let $X$ be a projective variety of dimension $n$ defined over an algebraically closed field and let $f: X \longrightarrow \mathbb P^n$ be a generically finite map. Let $L\subseteq \bP^n$ be a general line. Then $f^{-1}(L)$ is connected.
\end{proposition}

We start with the simplest case.
\subsubsection{Case $n=2$}\label{num:conn} Since the map $f$ is generically finite, this implies that $X$ is a surface.
We need some numerical connectedness considerations:

\begin{definition}
Let $X$ be a projective surface. A divisor $H$ is $1$-connected if for any non-trivial effective divisors $A$ and $B$ such that $A$ and $B$ do not have common components and $H = A + B$, then $A\cdot B \geq 1$.
\end{definition} 

The proof  of the following lemma is contained in Lemma 3.11, \cite{kollar}.

\begin{proposition}\label{1conn}
Let  $H$ be a big and nef divisor on a smooth projective surface $X$. Then $H$ is $1$-connected.
\end{proposition}

\begin{proof}
Let $H$ be a big and nef divisor and let $A$ and $B$ be non trivial effective divisors such that $H = A + B$. Since $X$ is projective, there exist very ample divisors on $X$, hence $A$ and $B$ are not numerically trivial. $H$ nef gives
\begin{equation*}
A^2 + A \cdot B = H \cdot A \geq 0
\end{equation*}
\begin{equation*}
A \cdot B + B^2 = H \cdot B \geq 0
\end{equation*}
Now if $A \cdot B \leq 0$ then $A^2B^2 \geq (A \cdot B)^2 \geq 0$ contradicts the Hodge Index Theorem, see Theorem \ref{genhodge}. Thus, $A \cdot B \geq 1$ and $H$ is 1-connected. 
\end{proof}

Let $X$ be a smooth projective surface, let $f:X\longrightarrow \mathbb P^2$ a dominant map and $L \subset \mathbb{P}^2$ a line. Let $H = f^*(L)$ be the pullback of $L$ through $f$. We claim that $H$ is big and nef. Indeed,  let $C$ be any irreducible curve on $X$, then by the projection formula 
\begin{equation*}
H\cdot C = f^*(L)\cdot C = L\cdot f_*(C) \geq 0.
\end{equation*}
Therefore, $H$ is nef. The selfintersection of $H$ is 
\begin{equation*}
H^2 =  f^*(L)\cdot f^*(L) = f_*(f^*(L)) \cdot L = deg(f)\,L^2  >0.
\end{equation*}
Thus, by Theorem \ref{bignessfornef} $H$ is big. 
Let us suppose by contradiction that the preimage of a line is not connected. Hence, there exist two effective non trivial divisors $A$ and $B$ such that $A + B = H$ and $A\cdot B = 0$. This contradicts Proposition \ref{1conn}.

\subsubsection{Proof of the Proposition \ref{line:conn}} 
It is enough to show how to reduce to the case $n=2$. Let $H \subseteq \bP^n$ be a general hyperplane. Let $D := f^*(H)$ be the preimage of $H$. The divisor $D \subset X$ might be non irreducible, so we write $D = \sum_ i n_iD_i$, where $D_i$ are all the irreducible components of $D$. Since $H$ is a general ample divisor, then $D$ is nef, big and semiample. Moreover, the restriction of $f^*(H)$ to $D_i$ gives a linear series without base points, because if there was a base point it would be also a base point of $|f^*(H)|$. Thus, the $D_i$ have to move algebraically and given any point $p$ we can find $D_j$ numerically equivalent to $D_i$ that does not have $p$ in its support. This implies that $f$ does not contract $D_i$: otherwise, since $D_i$ cover $X$, the fiber of the general point of $\bP^n$ would have dimension greater than one. We set $ f_*(D_i)=k_i H $ with  $k_i>0$.

We will use the following notation: $A=\sum_{i\in S} n_iD_i$ and $B=\sum_{i\in T}n_iD_i$ such that $S\cap T=\emptyset$ and  $A+B=D$. Then:
\begin{lemma} \label{lemma1}
With the notations above, $A \cdot D^{n - 1} > 0$ and $B \cdot D^{n - 1} > 0$.
\end{lemma}

\begin{proof}
It is enough to prove that $D_i \cdot  D^{n - 1} > 0$ for each $i$.
\begin{equation*}
D_i\cdot D^{n-1}=D_i\cdot f^\ast (H^{n-1})= f_\ast(D_i)\cdot H^{n-1}= k_i H^{n}=k_i>0.
\end{equation*}
\end{proof}

\begin{lemma} \label{lemma2}
With the above notations, the following holds:
\begin{equation*}
H^{n-2}\cdot f_\ast(A\cdot B)=f^{\ast}(H^{n-2}) (A\cdot B)=D^{n-2}\cdot A\cdot B>0
\end{equation*}
\end{lemma}

\begin{proof}
We assume, by contradiction, that $D^{n-2} \cdot A \cdot B = 0$. So we have
\begin{equation} \label{eq1}
(A + B)^{n-2} \cdot A \cdot B = \sum_{k=0}^{n-2}{n-2 \choose k}A^{k + 1}B^{n-1-k}=0
\end{equation}
Since $A$ and $B$ are nef, all the coefficients are positive and all the addends of the sum are zero. In particular 

\begin{enumerate}
\item[(i)] $A^{n-1} \cdot B = 0$ corresponding to $k=n-2$;

\item[(ii)] $A \cdot B^{n-1} = 0$ corresponding to $k=0$;
\end{enumerate}
Using now the generalized inequality of Hodge type, see \ref{genhodge}, we have the following
\begin{equation*} 
\begin{aligned}
0 = &(A \cdot B \cdot \underbrace{D \dots D}_{n-2})^{n-1} 
\\
&\geq (A^{n-1}\cdot D)(B^{n-1}\cdot D)\underbrace{(D^{n-1}\cdot D)\dots (D^{n-1}\cdot D)}_{n-2} \geq 0
\end{aligned}
\end{equation*} 
Since $D^n>0$ thanks to Theorem \ref{bignessfornef}, then either $A^{n-1}\cdot D=0$ or $B^{n-1}\cdot D=0$. We can assume to be in the first case, and the proof is the same for the second case. Using (i), we have that $A^{n-1}\cdot D=A^n + A^{n-1}\cdot B =A^n= 0$.  Therefore, $A \cdot D^{n-1}=\sum_{k=0}^{n-1}A^{k+1}B^{n-1-k}=0$, since the term corresponding to $k= n-1$ is $A^n$ and all the other addends are zero because they appear in (\ref{eq1}). This is in contradiction with Lemma \ref{lemma1}.
\end{proof}

Now we finish the proof of Proposition \ref{line:conn}. We proceed by induction on $n\ge 2$. The initial step has been proved in the previous Subsection \ref{num:conn}.
Let us consider a general line $L \subseteq \bP^n$ and let $H$ be a general hyperplane containing $L$. As before, we denote $f^*(H)=:D=\sum_ i n_iD_i$ and the restrictions of $f$ to $D_i$
\begin{equation*}
f_i: D_i \longrightarrow \bP^{n-1}= H
\end{equation*}
are dominant. Since $D_i$ might be singular we have to consider a regular alteration  $a_i:\tilde D_i \longrightarrow D_i$, where $\tilde D_i$ are smooth.
 By induction hypothesis,
we can assume that $a_i^{-1}f_{i}^{-1}(L)$ are connected curves. In particular $f_{i}^{-1}(L) = C_i$ are also connected (otherwise the preimages via $a_i$ of the components of $C_i$ would desconnect $a_i^{-1}f_{i}^{-1}(L)$). Let us set $C := \cup_i C_i$. We have to prove that $C$ is connected. Let us set $Z:= \cup_{i \in S}C_i$ and $Y:= \cup_{i \in T}C_i$ with $S \cap T = \emptyset$ and let $A$ and $B$ be two divisors on $X$, whose support is 
\begin{equation*}
[A]:= \cup_{i \in S}D_i  \\\ \text{    and    }\\\ [B]:= \cup_{i \in T}D_i. 
\end{equation*}
We will prove that $Z\cap Y \neq \emptyset$. We set $\Sigma:= [A] \cap [B] \subseteq X$, then $\Gamma := f(\Sigma)$ is a
 subvariety in $\bP^n$ such $\Gamma \cap H^{n-2}$ is a curve, otherwise, $H^{n-2}\cdot f_\ast(A\cdot B) = 0$ in contradiction with Lemma \ref{lemma2}. Since $H^{n-2}$ is a plane in $\bP^n$ that contains $L$, $\Gamma$ have to intersect $L$. Let $p \in \Gamma \cap L$, we can find a point $q\in X$ such that $f(q) = p$ and $q \in [A] \cap [B] \cap f^{-1}(L) = Z \cap Y$. Thus, $Z \cap Y \neq \emptyset$ and $f^{-1}(L)$ is connected.
\\

This concludes the proof of the Proposition.

\subsection{Connectedness of the preimage of a line for a dominant morphism}
We want now to prove the following proposition:

\begin{proposition}\label{line:conn_dom}
Let $X$ be a projective variety of dimension $n$ defined over an algebraically closed field and let $f: X \longrightarrow \mathbb P^n$ be a dominant morphism. Let $L\subseteq \bP^n$ be a general line. Then $f^{-1}(L)$ is connected.
\end{proposition}

\begin{proof}
Using the Stein Factorization of $f$, see \cite[III, Corollary 11.5]{hartshorne}, we obtain the following diagram:

\[
\xymatrixcolsep{5pc} \xymatrixrowsep{4pc}
\xymatrix{
X \ar[rd]^{f'} \ar[rr]^{f} &&
\mathbb{P}^n\\
& Y\ar[ru]^{g} }
\]
where $f'$ has connected fibers and $g$ is a finite map into $\mathbb{P}^n$. Let $\lambda: \tilde{Y} \longrightarrow Y$ be a regular alteration of $Y$.

\[
\xymatrixcolsep{5pc} \xymatrixrowsep{4pc}
\xymatrix{
X \ar[rd]^{f'} \ar[rr]^{f}  &&
\mathbb{P}^n\\
& Y \ar[ru]^{g} \\
& \tilde{Y} \ar[u]^{\lambda} \ar[ruu]^{\tilde{g}} }
\]
Since the map $\tilde{g}$ is generically finite and $\tilde{Y}$ is nonsingular, then using Proposition \ref{line:conn} we get that $\tilde{g}^{-1}(L)$ is connected, for any general line $L \subseteq \mathbb{P}^n$. Thus, $g^{-1}(L)$ is connected, otherwise $\tilde{g}^{-1}(L) = \lambda^{-1}g^{-1}(L)$ would not be connected. Since the fibres of $f'$ are connected, we get that $f^{-1}(L)$ is connected. 
This concludes the proof.

\end{proof}

\subsection{Connectedness of a higher dimension linear variety} \label{higherdim}
Finally, we have to consider the case of $L$ of arbitrary dimension.
Let $H\in(\mathbb{P}^n)^{\vee}$ be an hyperplane and let $p \notin H$ a general point. We consider the projection $\Pi: \mathbb{P}^n \setminus \{p\} \longrightarrow H$. We can identify $H$ with the projective space of the lines through $p$. Then the blow-up of $\mathbb{P}^n$ in the point $p$ can be seen as
\begin{equation*}
 B_p=\{(x,l)\in \mathbb P^n \times H \, \vert \, x\in l\}.
\end{equation*}
The exceptional divisor is $E_p=\{p\}\times H$ and the first projection gives
the birational map $\epsilon_p: B_p \longrightarrow \mathbb{P}^n$. Since we resolved the indeterminacy locus, the map $\epsilon_p \circ \Pi$ is now a morphism. 

We can now consider the fibre product $\tilde{X} := X \times_{\mathbb{P}^n} B_p$. We obtain a well-defined dominant morphism $f': \tilde X \longrightarrow \mathbb{P}^{n - 1} $. The fibre $f^{-1}(p)$ might be singular because of the failure of the Generic Smoothness Theorem in positive characteristic. Therefore, singularities might arise in $\tilde{X}$. As before,
we consider a regular alteration of $\tilde{X}$, $f^{''}: X^{''} \longrightarrow \tilde{X}$, where $X^{''}$ is a smooth irreducible variety.

\[\xymatrixcolsep{5pc} \xymatrixrowsep{4pc}
\xymatrix{
X^{''}\ar[d]^{f^{''}} \\
\tilde{X} \ar[r]^{\tilde{f}} \ar[d]^{\tilde{\epsilon}} \ar[rdd]^{f'} & B_p \ar[d]^{\epsilon_p} \\
X \ar[r]_f & \mathbb{P}^n \ar@{-->}[d]^{\Pi} \\
& \mathbb{P}^{n - 1} }
\]
We call $h:= f'' \circ f'$. Assuming the connectedness statement for $h$, we get that $h^{-1}(L)$ is connected for any linear variety $L\subseteq \bP^{n-1}$ of dimension $k \geq 1$. Then we get that $f''( h^{-1}(L))$ is connected too, otherwise the preimages via $f''$ of the components would disconnect $h^{-1}(L)$. Since $f^{-1}(L\vee p)=
\tilde \epsilon(f''( h^{-1}(L))$, we get that the connectedness statement is true for any linear variety of $\bP^{n}$ containing $p$ of dimension $k' \geq 2$. To obtain the thesis for every linear variety $L$ contained in $\bP^{n}$ of dimension $k' \geq 2$, it is sufficient to consider a point $p\in L$ and repeat the previous construction. Observe that with this procedure we do not reach the one dimensional case that we have proved in Proposition \ref{line:conn_dom}. This concludes the proof of Theorem \ref{main}.

\section{Classic Bertini Theorem} \label{classicalBertini}

If  in Theorem \ref{main} we consider the inclusion map, we get the classic version of Bertini Theorem. The theorem states that a general hyperplane section of a non singular variety in a projective space is again nonsingular and connected if the dimension of the variety is greater than two. It holds over an arbitrary algebraically closed field of any characteristic and can be proved by computing the dimension of non transverse hyperplanes to the variety (see \cite[Theorem II.8.18]{hartshorne}).

\begin{theorem}[Bertini theorem]
Let  X be a nonsingular closed subvariety of $\mathbb{P}^n_k$, where {\itshape k} is an algebraically closed field. Then there exists an hyperplane $H \subseteq \mathbb{P}^n_k$, not containing $X$, and such that the scheme $H \cap X$ is regular at every point. If $\dim X  \ge 2$, $ H \cap X$ is also connected. Furthermore, the set of hyperplanes with this property forms an open dense subset of the complete linear system $|H|$, considered as projective space.
\end{theorem}

\section{Connectedness theorem of Fulton and Hansen} \label{fultonhansen}

We give a very short proof of the connectedness Theorem of Fulton and Hansen where we re-elaborate an idea of Deligne.

\begin{theorem}\label{Fulton-Hansen}
Let $X$ be an irreducible projective variety, and
\begin{equation*}
F: X \longrightarrow \bP^{n} \times \bP^{n}
\end{equation*}
a morphism. Assume that $\dim F(X)>n$. Let $\Delta \subseteq \bP^{n} \times \bP^{n}$ be the diagonal, then the inverse image $F^{-1}(\Delta) \subseteq X$ of the diagonal  is connected.
\end{theorem}

\begin{proof}
Let us consider $\bP^{n} \times \bP^{n}$ with homogenous coordinates $[z_i]$ and $[w_j]$, let $f: X \longrightarrow \mathbb{P}^n$ and $g: X \longrightarrow \mathbb{P}^n$ be morphisms such that $F = (f, g)$. Then we can consider the following line bundles on $X$: 
\begin{equation*}
L:= f^* \mathcal{O}_{\bP^n}(1),\\\ \\\ M := g^* \mathcal{O}_{\bP^n}(1)
\end{equation*}
with sections $s_i = f^*(z_i)$ and $t_j = g^*(w_j)$. Let
\begin{equation*}
E:= L \oplus M.
\end{equation*}
be the rank two bundle and let $\pi: \mathbb{P}(E) \longrightarrow X$ be the associated projective space bundle. We have the tautological exact sequence:
\[
 \xymatrix@C=1.pc@R=1.8pc{
0 \ar[r]     &\mathcal N \ar[r] &\pi^*(E) \ar[r]^{\phi } &  \mathcal{O}_{E}(1) \ar[r] & 0 }
\]
where $\mathcal{N}$ is defined as the kernel of the map $\phi$. The bundle $\pi^*(E)$ is generated by the $2n + 2$ sections $\pi^*(s_i)$ and $\pi^*(w_j)$, and their image via the map $\phi$ generate $\mathcal{O}_{E}(1)$. Therefore, we have the map $Q: \mathbb{P}(E) \longrightarrow \bP^{2n + 1}$, associated to $\pi^*(E)$. In coordinates: 
\begin{equation*}
\begin{aligned}
P \longmapsto &(\pi^*s_0(P), \dots, \pi^*s_n(P), \pi^*t_0(P), \dots, \pi^*t_n(P)).
\end{aligned}
\end{equation*}
We remark that 
\begin{equation*}
\dim Q(\bP(E)) = \dim F(X) + 1.
\end{equation*}
If we consider the embedding of  $\bP^{n} \times \bP^{n}$ into $\bP^{2n + 1}$, the image of the diagonal $\Delta$ is given by the $n$-dimensional linear subspace $L \subseteq \bP^{2n + 1}$ defined be the equations $\{z_i = w_i\}$. Then 
\begin{equation*}
F^{-1}(\Delta) = \pi(Q^{-1}(L))
\end{equation*} and the thesis follows from Theorem \ref{main}. 
\end{proof}

From Theorem \ref{main} we can deduce a connectedness theorem for flag manifolds and grassmannians.

\begin{theorem}
Let $\mathbb{F}$ be any flagmanifold in $\bP^{n}$, and $\Delta_F$ be the image of the diagonal embedding of $\mathbb{F}$ in $\mathbb{F} \times \mathbb{F}$. Let $X$ be an irreducible variety and $f: X \longrightarrow \mathbb{F} \times \mathbb{F}$ a morphism. Then $f^{-1}(\Delta_F)$ is connected if codim $(f(X),\mathbb{F} \times \mathbb{F}) < n$.
\end{theorem}

Is is enough to follow the construction  in \cite{hansen} and then apply Theorem \ref{Fulton-Hansen}.


\begin{bibdiv}
 \begin{biblist*}

\bib{bertini}{article}{
   author={Bertini, Eugenio},
   title={Sui sistemi lineari},
   journal={Istit. Lombardo Accad. Sci. Lett. Rend. A Istituto},
   volume={15},
   date={1882},
   pages={24--29},
}


\bib{Deligne}{article}{
   author={Deligne, Pierre},
   title={Le groupe fondamental du compl{\'e}ment d'une courbe
   plane n'ayant que des points doubles ordinaires est ab{\'e}lien},
   journal={S{\'e}minaire Bourbaki,},
   number={543},
   date={Nov.1979},

}

\bib{DeJong}{article}{
   author={De Jong, Aise Johan},
   title={Smoothness, semi-stability and alterations},
   journal={Inst. Hautes Études Sci. Publ. Math.},
   number={83},
   date={1996},
   pages={51--93},
}

\bib{fulthansen}{article}{
   author={W. Fulton; J. Hansen},
   title={A Connectedness Theorem for Projective Varieties, with Applications to Intersections and Singularities of Mappings},
   journal={Annals of Mathematics},
   volume={110},
   number={1},
   date={Jul., 1979},
   pages={159--166},
}

\bib{fultlaz}{book}{
   author={ W. Fulton; R. Lazarsfeld},
   title={Connectivity and its applications in algebraic geometry},
   note={Algebraic Geometry (Chicago, 1980), 26-92, LNM 862},
   publisher={Springer},
   number={1},
   place={Berlin - New York},
   date={1981},
}


\bib{hansen}{article}{
   author={J. Hansen},
   title={A Connectedness Theorem for Flagmanifolds and Grassmannians},
   journal={American Journal of Mathematics},
   volume={105},
   number={3},
   date={Jun., 1983},
   pages={633--639},
}

\bib{hartshorne}{book}{
   author={Hartshorne, Robin},
   title={Algebraic geometry},
   note={Graduate Texts in Mathematics, No. 52},
   publisher={Springer-Verlag},
   place={New York},
   date={1977},
}


\bib{jouanolou}{book}{
   author={Jouanolou, J. P.},
   title={Theoremes de Bertini et applications},
   publisher={Progress in Mathematics},
   place={Birkhuser Boston, Boston, MA},
   date={1983},
}



\bib{Kleiman}{article}{
   author={ Kleiman, S.},
   title={Bertini and his two fundamental Theorems},
   journal={Rend. Circ. Mat. Palermo},
  
}





\bib{kollar}{book}{
   author={Koll{\'a}r, J{\'a}nos},
   title={Complex Algebraic geometry},
   note={IAS, Park City Mathematics Series, Vol. 3},
   publisher={American Mathematical Society Institute for Advanced Study},
   place={Park City},
   date={1997},
}


\bib{lazarsfeld}{book}{
   author={Lazarsfeld, Robert},
   title={Positivity in Algebraic geometry},
   note={Graduate Texts in Mathematics, No. 52},
   publisher={Springer},
   place={New York},
   date={2000},
}




\bib{seidenberg}{article}{
   author={Seidenberg, A.},
   title={The Hyperplane Sections of Normal Varieties},
   publisher={Transactions of the American Mathematical Society},
   volume={69}
   number={2}
   date={1950},
   pages={327--386},
} 


\bib{sommese}{article}{
   author={A. Sommese; A. Van de Ven},
   title={Homotopy groups of pullbacks of varieties},
   publisher={Nagoya Math. J.},
   volume={102}
   date={1986},
   pages={79--90},
} 



\bib{Weil}{book}{
   author={Weil, A.},
   title={Foundations of Algebraic Geometry },
   publisher={American Mathematical Society},
   place={Providence (Rhode Island)},
   date={1962},
} 


\bib{zak}{article}{
   author={Zak, F. L.},
   title={Tangents and Secants of Algebraic Varieties},
   publisher={Transl. Math. Monogr. Amer. Math. Soc.},
   volume={127},
   date={1993},
} 



\end{biblist*}
\end{bibdiv}
\end{document}